\documentclass{amsart}
\usepackage{amsmath}
\usepackage{amssymb}
\usepackage{amsthm}
\usepackage{amsfonts}
\usepackage{mathtools}
\usepackage{url}
\usepackage{hyperref}
\usepackage{appendix}
\usepackage{verbatim}
\usepackage{graphicx,color}
\usepackage{tikz-cd}
\usepackage{thmtools}
\usepackage{thm-restate}
\usepackage{cleveref}

\title{On the profinite rigidity of triangle groups}
\author{M. R. Bridson}
\author{D. B. McReynolds}
\author{A. W. Reid}
\author{R. Spitler}
\address{\newline Mathematical Institute, 
\newline Andrew Wiles Building,
\newline University of Oxford,
\newline Oxford OX2 6GG, UK}
\email{ bridson@maths.ox.ac.uk}
\address{\newline Department of Mathematics,
\newline Purdue University
\newline West Lafayette, IN 47907, USA}
\email{dmcreyno@purdue.edu}
\address{\newline Department of Mathematics,
\newline Rice University, 
\newline Houston, TX 77005, USA}
\email{ alan.reid@rice.edu}
\address{\newline Department of Mathematics,
\newline McMaster University
\newline Hamilton, Ontario, Canada}
\email{spitlerr@mcmaster.ca}

\def\-{\overline}

\def\wh{\widehat}

\def\G{\Gamma}

\def\H{\mathbb{H}}
\def\Z{\mathbb{Z}}
\def\R{\mathbb{R}}
\def\Q{\mathbb{Q}} 
\def\F{\mathbb{F}} 
\def\C{\mathbb{C}}

\def\tr{\mbox{\rm{tr}}\, }
\def\P{\mbox{\rm{P}}}
\def\O{\mbox{\rm{O}}}
\def\G{\Gamma}
\def\D{\Delta}

\def\<{\langle}
\def\>{\rangle}
\def\ilim{\varprojlim}

\def\Hom{{\rm{Hom}}}

\def\onto{\twoheadrightarrow}

\newcommand{\innp}[1]{\left< #1 \right>}
\newcommand{\abs}[1]{\left\vert#1\right\vert}
\newcommand{\set}[1]{\left\{#1\right\}}

\DeclareMathOperator{\SL}{SL} \DeclareMathOperator{\PSL}{PSL}

\newtheorem{theorem}{Theorem}[section]
\newtheorem{lemma}[theorem]{Lemma}
\newtheorem{corollary}[theorem]{Corollary}
\newtheorem{proposition}[theorem]{Proposition}

\theoremstyle{definition} 
\newtheorem{definition}[theorem]{Definition}
\newtheorem{remark}[theorem]{Remark}

\begin{document}


\begin{abstract}  
We prove that certain Fuchsian triangle groups are profinitely rigid in the absolute sense, i.e.~each is distinguished from all other finitely generated, residually finite groups by its set of finite quotients. We also develop a method based on character varieties that can be used to distinguish between the profinite completions of certain groups.
\end{abstract}

\maketitle

\centerline{\em{For Bill Harvey, who knew Fuchsian groups well}}

%
%
\section{Introduction}\label{intro} 

A finitely generated, residually finite group $\G$ is {\em profinitely rigid (in the absolute sense)} if it is distinguished from all other finitely generated, residually finite groups by its set of finite quotients. More formally, if $\Lambda$ is finitely generated and residually finite, then $\wh{\Lambda}\cong \wh{\G}$ implies ${\Lambda}\cong {\G}$ (where $\wh{\Delta}$ denotes the profinite completion of a group $\Delta$). Finitely generated abelian groups have this property, as do certain nilpotent groups, but it is hard to construct examples of profinitely rigid groups that do not satisfy a group law; indeed no such groups were known until our work in \cite{BMRS}. The most compelling question in the field is the conjecture that non-abelian free groups of finite rank are profinitely rigid. More generally, it seems reasonable to expect all lattices in $\PSL(2,\mathbb{R})$ to be profinitely rigid in this sense. The main result of \cite{BCR} shows that such lattices can at least be distinguished from each other by their finite quotients. 

In \cite{BMRS} we proved that certain arithmetic lattices in $\PSL(2,\mathbb{C})$ are profinitely rigid, including the Bianchi group $\PSL(2,\Z[\omega])$ (where $\omega^2+\omega+1=0$) and the fundamental group of the Weeks manifold, which is the closed hyperbolic 3--manifold of minimal volume. The main purpose of the present article is to prove that certain arithmetic lattices in $\PSL(2,\mathbb{R})$ are also profinitely rigid in the absolute sense.  

\begin{restatable}{thm}{main}\label{main}
The following arithmetic triangle groups are profinitely rigid in the absolute sense:
\begin{align*}
&\Delta(3,3,4), \Delta(3,3,5),\ \Delta(3,3,6),\ \Delta(2,5,5), \Delta(3,5,5), 
 \Delta(4,4,4), \Delta(5,5,5),\\
&\Delta(2,3,8), \Delta(2,3,10), \Delta(2,3,12), \Delta(2,4,5), \Delta(2,5,6), 
\Delta(2,4,8), \Delta(2,5,10). 
\end{align*}
\end{restatable}

We follow the convention of writing $\Delta(p,q,r)$ to denote the orientation-preserving subgroup of index $2$ in the Coxeter group generated by reflections in the sides of a hyperbolic triangle with vertex angles $\pi/p$, $\pi/q$, and $\pi/r$. We shall also prove (Corollary \ref{c:non-orient}) that for the above values of $(p,q,r)$, these Coxeter groups are profinitely rigid in the absolute sense.

In its broad outline, our strategy for proving Theorem \ref{main} is the one employed in \cite{BMRS} to establish the existence of Kleinian groups that are profinitely rigid. Several of the key ideas developed in \cite{BMRS}, notably that of {\em Galois rigidity}, will play a crucial role again here. But the end-game by which we move from the general construction of representations to profinite rigidity for specific examples is more direct in the present setting than it was in
\cite{BMRS}. 

Roughly speaking, a finitely generated subgroup $\G<{\rm{PSL}}(2,\C)$ is Galois rigid if all of the irreducible representations $\G\to {\rm{PSL}}(2,\C)$ can be constructed from the arithmetic of the trace field of $\G$ (see \S \ref{s:galois_rigid}). The first five groups in Theorem \ref{main} arise as the image in $\PSL (2,\C)$ of the elements of norm $1$ in a maximal order in a quaternion algebra over a real quadratic field, ramified at one infinite place and one finite place; see Theorem \ref{arithtriangle}. Each of these groups is Galois rigid; see Proposition \ref{trianglereps}. Both of $\Delta(4,4,4)$, and $\Delta(5,5,5)$ arise as a subgroup of index $3$ in one of the previous groups; these are also Galois rigid. The other seven groups in Theorem \ref{arithtriangle} are index $2$ extensions of the groups written above them.

In the proof of Theorem \ref{main}, we use classical properties of Fuchsian groups to prove that each proper infinite  subgroup of the triangle groups $\Delta$ has a finite quotient that $\Delta$ does not have (see Corollary \ref{c:whole}). This can also be deduced from the more widely applicable criterion that we develop in Section 5 using the theory of character varieties. In addition, in \S \ref{s:character} we present a more general method using the structure of character varieties to produce homomorphisms to subgroups of ${\rm{PSL}}(2,\mathbb{F})$ (for a finite field $\F$) that can also be used to distinguish between the profinite completions of the groups that interest us. Recall that a finitely generated group $\G$ has Property FA if it has a global fixed point whenever it acts on a simplicial tree.

\begin{restatable}{thm}{charvar}\label{t:character-var}
Let $\G < \PSL(2,\C)$ be a finitely generated Kleinian group that has Property FA. Suppose that $H$ is a finitely generated non-elementary Kleinian group whose  $\PSL(2,\C)$--character variety has an irreducible component of positive dimension containing the character of a discrete faithful representation of $H$. Then there exists a finite field $\F$ and a representation $H\to\PSL(2,\F)$ whose image is not a quotient group of $\G$.
\end{restatable} 

\noindent{\bf{Acknowledgements:}} We acknowledge with gratitude the financial support of the Royal Society (MRB) and the National Science Foundation (DBM and AWR). In addition, we thank the Hausdorff Institute in Bonn, ICMAT in Madrid and the University of Auckland for their hospitality during the writing of this article.
 
\section{Trace-fields and Galois Rigidity}\label{s:BMRS}

\noindent We recall some of the key ideas from  \cite{BMRS}.

\subsection{Trace-fields}\label{ss:traces}

To fix notation, it will be convenient to record some basic facts about trace-fields of finitely generated subgroups of $\PSL(2,\C)$.  Let $\phi\colon \mathrm{SL}(2,\mathbb{C}) \to \mathrm{PSL}(2,\mathbb{C})$ be the quotient homomorphism, and $H$ a finitely generated subgroup of $\PSL(2,\mathbb{C})$. Let $H_1 = \phi^{-1}(H)$. It will be convenient to say $H$ is {\em Zariski dense} in $\PSL(2,\C)$ when what we actually mean is that $H_1$ is a Zariski dense subgroup of $\SL(2,\C)$. The \textit{trace-field} of $H$ is defined to be the field 
\[ K_H=\mathbb{Q}(\mathrm{tr}(\gamma)~\colon~ \gamma \in H_1). \] 
If $K_H$ is a number field with ring of integers $R_{K_H}$, we say that $H$ has {\em integral traces} if $\tr(\gamma)\in R_{K_H}$ for all $\gamma \in H_1$. The group $H_1$ generates a $K_H$--quaternion algebra $A_0H$, and when $H$ has integral traces, $H_1$ generates an $R_{K_H}$--order $\mathcal{O}H$ in $A_0H$ (see \cite[Ch 3]{MR} for more details on this material). Conversely, if $H_1$ is contained in an order of $A_0H$, then $H$ has integral traces.  

One well-known situation when $K_H$ is a number field is the following (see \cite{BMRS} for example). Let $\mathrm{X}_{\mathrm{zar}}(H,\C)$ denote the set of Zariski dense representations up to conjugacy of $H$ in $(\P)\SL(2,\C)$.

\begin{lemma} \label{trace_field_number_field}
Let $H< \PSL(2,\C)$ be a finitely generated group. If $\mathrm{X}_{\mathrm{zar}}(H,\C)$ is finite, then $K_H$  is a number field.
\end{lemma}

Suppose that $H$ is a finitely generated group and  $\rho\colon H\rightarrow (\P)\SL(2,\C)$ a Zariski dense representation with $K=K_{\rho(H)}$ a number field of degree $n_K$. If $K=\Q(\theta)$ for some algebraic number $\theta$, then the Galois conjugates of $\theta$, say $\theta=\theta_1,\dots,\theta_{n_K}$ provide embeddings $\sigma_i\colon K\to\C$ defined by $\theta\mapsto\theta_i$.  These in turn can be used to build $n_K$ Zariski dense non-conjugate representations $\rho_{\sigma_i}\colon H \to (\P)\SL(2,\C)$ with the property that $\tr(\rho_{\sigma_i}(\gamma))=\sigma_i(\tr\rho(\gamma))$ for all $\gamma\in H$. We sometimes refer to these as {\em Galois conjugate representations}.

\subsection{Galois rigidity}\label{s:galois_rigid}

As in \cite{BMRS} we will be interested in groups $\G$ with the fewest possible Zariski dense representations. To be more precise, recall from the discussion in \S \ref{ss:traces} that if $\G$ is a finitely generated group and $\rho\colon \G\to (\P)\SL(2,\C)$ a Zariski dense representation whose trace field $K_{\rho(\G)}$ is a number field, then using the Galois conjugate representations, we have $\abs{\mathrm{X}_{\mathrm{zar}}(\G,\mathbb{C})}\geq  n_{K_{\rho(\G)}}$.  

\begin{definition}[Galois rigid]\label{def:galois-rigid}
Let $\G$ be a finitely generated group and let $\rho\colon \G\to (\P)\!\SL(2,\C)$ be a Zariski dense representation whose trace field $K_\G$ is a number field. If $\abs{\mathrm{X}_{\mathrm{zar}}(\G,\mathbb{C})}= n_{K_{\rho(\G)}}$, we say that $\G$ is {\em Galois rigid}  (with associated field $K_{\rho(\G)}$).
\end{definition}

When we say that a subgroup of $(\P)\!\SL(2,\C)$ is Galois rigid, we are implicitly taking $\rho$ to be the inclusion map. In this case $\G$ is Galois rigid if and only if $\rho(\G)$ is Galois rigid. Note too that if $\G<(\P)\!\SL(2,\C)$ is Galois rigid, then any irreducible representation with infinite image can serve as $\rho$, as any such representation is a Galois conjugate of any other. In particular, $K_\G$, considered as an abstract field, is an intrinsic invariant of $\G$, as is the associated quaternion algebra $A_0\G$.

The theorem stated below can be extracted from \cite[Thm 4.7, Cor 4.10]{BMRS}, but the special case that we require involves some preliminaries. To that end, we fix a real quadratic number field $K\subset \R$ with $\sigma\colon K\to \R$ the non-trivial Galois embedding, a quaternion algebra $B/K$, and a maximal order $\mathcal{O} < B$. Since $K$ has two real places $v_1$ (the identity place) and $v_2$ (associated to $\sigma$), we can prescribe that $B$ be ramified at either of $v_1$ or $v_2$, and unramified at the other; denote these two possibilities by $B_1$ and $B_2$ respectively. If these $B_i$ are only additionally ramified at a finite place $\omega$ with residue field of characteristic $p$ where $\omega$ is the unique such place, then although $B_1$ and $B_2$ are not isomorphic (over $K$), there is an extension of $\sigma$ that maps $B_1$ to $B_2$. In this situation, and up to this ambiguity, we can identify $B$ with either of the $B_i$ (for $i=1,2$) and assume that $\Gamma < \mathcal{O}^1$ is a finitely generated subgroup such that $K_\G= K$. We identify $A_0\G$ with $B$.  With this preamble established, the following result can readily be extracted from  \cite[Thm 4.7, Cor 4.10]{BMRS}.

\begin{theorem}\label{galois_rigid_real_quad}
Let $\Gamma<\mathcal{O}^1<B$ be as above and assume that $\G$ is Galois rigid. If $\Delta$ is a  finitely generated, residually finite group with $\widehat{\Delta} \cong \widehat{\Gamma}$, then
\begin{itemize}
\item[(i)]
$\Delta$ is Galois rigid with associated field $K$.
\item[(ii)] 
If $B$ has type number $1$, and if $\mathrm{Ram}(B) = \set{v_2,\omega}$ where $v_2$ is the real place described above, and $\omega$ is a finite place as above, then there is a Zariski dense homomorphism $\phi'\colon \Delta \to \mathcal{O}^1$.
\end{itemize}
\end{theorem}

\section{Triangle groups}\label{s:arithtriangle}

For positive integers $p,q,r$ with $\frac 1 p + \frac 1 q + \frac 1 r <1$, we write $\D^-(p,q,r)$ for the group of isometries of the hyperbolic plane generated by reflections in the sides of a triangle $T(p,q,r)$ with vertex angles $\pi/p$, $\pi/q$, and $\pi/r$, and we write $\D(p,q,r)$ for the index 2 subgroup consisting of orientation preserving isometries. There are standard presentations
\begin{align*}
\D(p,q,r) &= \innp{ a,b,c \mid a^p, b^q, c^r, abc \> = \<a,b \mid a^p, b^q, (ab)^r }, \\
\D^-(p,q,r) &= \innp{ x,y,z \mid x^2, y^2, z^2, (xy)^p, (yz)^q, (xz)^r },
\end{align*}
where $a,b,c$ are rotations (same orientation) at the different vertices of $T(p,q,r)$. 

\subsection{Index 2 extensions of triangle groups}\label{index2}

As we try to establish profinite rigidity for the groups in Theorem \ref{main}, we will need to analyse index 2 extensions of certain triangle groups.   

\begin{proposition} \label{prop_index2}
Let $\D<\G$ be a pair of groups with $\D=\D(p,q,r)$ and $[\G:\D]=2$.
\begin{enumerate}
\item 
If $p,q,r$ are distinct then $\G\cong \D\times \Z/2\Z$ or $\G\cong \D^-(p,q,r)$.
\item 
If $q=r$, then there four possibilities for $\G$, up to isomorphism, all of the form $\D\rtimes\Z/2$; they are $\D(p,q,q)\times \Z/2\Z,\, \D^-(p,q,q),\, \D(2p,q,2)$ and 
\[  \Lambda_\rho(p,q) := \< a,b,c,\rho \mid 1=a^p=b^q=c^q=\rho^2=abc, \rho a\rho=a^{-1}, \rho b\rho=c, \rho c\rho=b\>.  \]
\end{enumerate}
\end{proposition}

\begin{proof} 
Kerckhoff's solution to the Nielsen Realisation Problem (see \cite[\S V, Thm 7]{Ker}) implies that for any group $G$ containing $\D$ as a subgroup of index $m$ there is a short exact sequence
\[ 1\to K \to G \to \Lambda \to 1\]
where $\Lambda$ is a lattice in ${\rm{Isom}}(\H^2)$ containing $\D$ as a subgroup of index $m/|K|$. In our setting $m=2$, so either $K=\Z/2\Z$ and $\Lambda = \Delta$, or else $K$ is trivial and $\Lambda$ contains $\Delta$ as a subgroup of index $2$. In the former case, the short exact sequence splits and $G=\Delta\times\Z/2\Z$. In the latter case, there is a 2--sheeted covering of Riemannian orbifolds $\H^2/\Delta \to \H^2/\Lambda$. In order to understand the possibilities for $\Lambda$, we analyse the possibilities for the deck transformation of this covering, i.e. the isometric involutions of the orbifold $\O=\H^2/\Delta$.

In each case, such an involution has at least one fixed point (one of the three cone points), so the corresponding outer automorphism of $\D=\pi_1^{\rm{orb}}\O$ lifts to an involution in ${\rm{Aut}}(\D)$ and $\Lambda$ is a semi-direct product of the form $\D\rtimes \Z/2\Z$. If $p,q,r$ are all distinct, then any isometry of $\O$ must fix the three cone points and the geodesic arcs joining them, i.e. the image in $\O$ of the boundary $\partial T$ of the geodesic triangle $T(p,q,r)\subset\H^2$. Thus the only non-trivial isometry of $\O$ in this case is the reflection in the image of $\partial T$. This reflection, which we denote by $\tau$, is the isometry of $\O$ induced by each of the three isometries of $\mathbb{H}^2$ that are reflections in the sides of $T(p,q,r)$. Thus $\Lambda \cong \D^-(p,q,r)$. This proves (1).

To prove (2), if $q=r\neq p$, then each isometry of $\O$ must fix the cone point with vertex angle $2\pi/p$, but it can interchange the other two cones points $v,v'$. If $\rho\colon \H^2\to\H^2$ is the reflection in the perpendicular bisector of the edge of $T$ opposite the vertex with angle $\pi/p$, then the induced isometry ${\rho_0}\colon \O\to \O$ interchanges $v$ and $v'$ and commutes with $\tau$. The product $\sigma=\tau\rho_0$ is the only other non-trivial isometry of $\O$; it lifts to a rotation $s$ through $\pi$ about the midpoint of the edge of $T(p,q,q)$ opposite the vertex with angle $\pi/p$. Note that $\<\D, s\> = \Delta(2p,q,2)$.

Finally, if $p=q=r$ then ${\rm{Isom}}(\O) = {\rm{Sym}}(3)\times\Z/2\Z$ where the second factor is generated by $\tau$. The distinct conjugacy classes of involutions are represented by $\tau, \rho$ and $\sigma$, so up to isomorphism (equivalently, conjugacy in ${\rm{Isom}}(\H^2)$) the possibilities for $\Lambda$ are the same as in the previous case.

Our description of explicit lifts for $\tau, \rho_0$ and $\sigma$ provides explicit presentations for the three possible lattices containing $\D$ as a subgroup of index $2$ in the case $q=r$ (regardless of whether $p=q$). For example, the reflection $\rho$ conjugates the generator $a\in\Delta(p,q,q)$ to its inverse while interchanging $b$ and $c^{-1}$. Thus  $\Lambda_\rho(p,q) = \< \D, \rho \>$ is given by the presentation in the statement of the proposition.\end{proof}

\subsection{Some arithmetic triangle groups}\label{atriangle}

In Theorem \ref{arithtriangle} we list a subset of the
arithmetic triangle groups whose invariant trace-field is a real quadratic $\Q(\sqrt{d})$ (the value of $d$ is given in the statement of the theorem); this list is taken from \cite{Tak}. In each case, the invariant quaternion algebra is ramified at one real place of $\Q(\sqrt{d})$ and one finite place (this is listed as $\mathcal{P}_q$ where $q$ is the rational prime with $\mathcal{P}_q|\langle q\rangle$). Each of the triangle groups arises as the image of the norm $1$ elements in a maximal order of the invariant quaternion algebra. Figure 1 portrays the degrees of commensurabilities of the triangle groups under consideration that will be useful later (taken from \cite{Tak}).

\begin{theorem}\label{arithtriangle} With the notation established above, we have:
\begin{enumerate}
\item $\Delta(3,3,4)$, $d=2$, $\mathcal{P}_2$.
\item $\Delta(3,3,6)$, $d=3$, $\mathcal{P}_2$.
\item $\Delta(2,5,5)$, $d=5$, $\mathcal{P}_2$.
\item $\Delta(3,5,5)$, $d=5$, $\mathcal{P}_3$.
\item $\Delta(3,3,5)$, $d=5$, $\mathcal{P}_5$.
\end{enumerate}
\end{theorem}

\noindent Note that the commensurability classes of $\Delta(2,3,8)$, $\Delta(2,3,12)$ and $\Delta(2,4,5)$ contain additional triangle groups for which our methods do not establish profinite rigidity (see Remark \ref{rem:notgr} below).

\[ 
\quad \begin{tikzcd}
& \Delta(2,3,8) \arrow[dash,"2" ']{dl} \arrow[dash, "3"]{dr} & \\
\Delta(3,3,4) \arrow[dash,"3"']{dr} & & \Delta(2,4,8) \\
& \Delta(4,4,4) \arrow[dash,  "2"']{ur}&
\end{tikzcd} 
\quad\quad
\begin{tikzcd}
& \Delta(2,3,10) \arrow[dash,"2"']{dl} \arrow[dash, "3"]{dr} & \\ \Delta(3,3,5) \arrow[dash,"3"']{dr} & & \Delta(2,5,10) \\ & \Delta(5,5,5) \arrow[dash,  "2"']{ur}&
\end{tikzcd} 
\]

\[\begin{tikzcd}
\Delta(2,3,12) \arrow[dash,"2"']{dd} \\ \\ \Delta(3,3,6)
\end{tikzcd} 
\quad\quad\quad\quad\quad\quad
\begin{tikzcd} 
 \Delta(2,4,5) \\ \\ \Delta(2,5,5) \arrow[uu,dash,"2"] 
\end{tikzcd}
\quad\quad\quad\quad\quad\quad
\begin{tikzcd} 
\Delta(2,5,6) \\ \\ \Delta(3,5,5) \arrow[uu,dash,"2"] 
\end{tikzcd} \]

\medskip

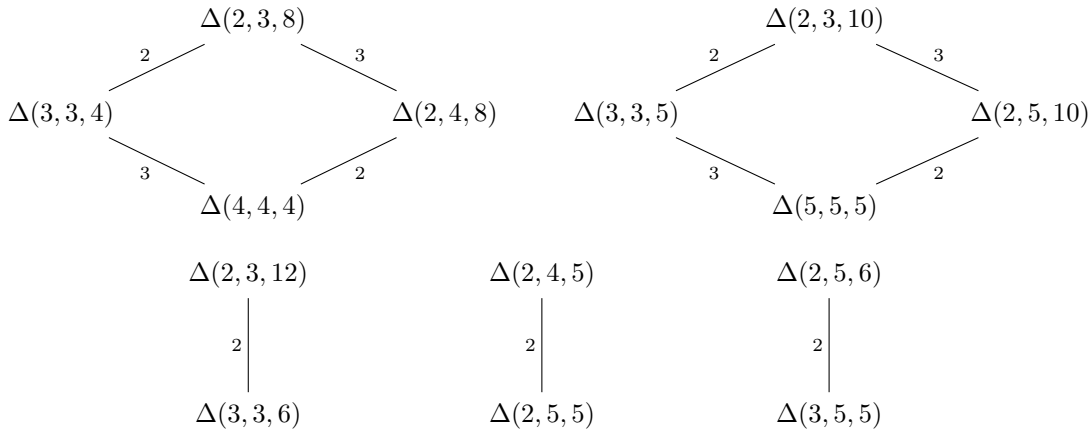
\centerline{\bf Figure 1}

\subsection{Galois rigidity of certain triangle groups}

\begin{proposition}\label{trianglereps}
Let $\Delta$ be one of the triangle groups $\Delta(3,3,4)$, $\Delta(4,4,4)$, $\Delta(3,3,6)$, $\Delta(2,5,5)$, $\Delta(3,5,5)$, $\Delta(3,3,5)$ or $\Delta(5,5,5)$. Then $\Delta$ is Galois rigid.
\end{proposition}

\begin{proof} 
We begin with some general comments. Let $\Delta=\Delta(p,q,r)$ and suppose that $\rho\colon \Delta \rightarrow \PSL(2,\C)$ is a non-trivial representation, then the orders of $\rho(a)$, $\rho(b)$ and $\rho(ab)$ divide $p$, $q$ and $r$ respectively.  Since we are only interested in irreducible representations we can conjugate $\rho$ so that 
\[ \rho(a) = \begin{pmatrix} \pm \zeta_p & 1\cr 0& \pm 1/\zeta_p\cr\end{pmatrix}, ~~ \rho(b) = \begin{pmatrix} \pm \zeta_q & 0\cr z & \pm 1/\zeta_q\cr\end{pmatrix}, \]
for some $z\in \C$ with $\zeta_p$ and $\zeta_q$ being $p$-th and $q$-th roots of unity respectively. Moreover, $z$ is constrained by the requirement that $\tr(\rho(ab)) = \zeta_p\zeta_q+1/(\zeta_p\zeta_q)+z = \pm (\zeta_r+1/\zeta_r)$ for some $r$-th root of unity. Visibly, there are only finitely many possibilities for $z$. To go further, and establish Galois rigidity, we must analyse the possible solutions.

We do the case of $\Delta(3,3,6)$ in some detail, the others are similar.  As above,  let $\rho\colon \Delta \to \PSL(2,\C)$ be an irreducible representation. Note that $\rho(a)$ and $\rho(b)$ must be elements of order $3$, whilst the possibilities for the order of $\rho(ab)$ are $2$, $3$ and $6$. If $\rho(ab)$ has order $2$, the image of $\rho$ is the alternating group $A_4$; in particular it is finite. If $\rho(ab)$ has order $3$, the image is the Euclidean triangle group $\Delta(3,3,3)$, and this can be conjugated to lie in the image in $\PSL(2,\C)$ of the upper triangular matrices; i.e.~it is a reducible representation. Hence $\rho(ab)$ has order $6$.  

To deal with this last case,  as above, we conjugate so that $\rho(a)$ and $\rho(b)$ have the following form (where $\omega^2+\omega+1=0$):
\[ \rho(a) = \begin{pmatrix} \pm \omega & 1\cr 0& \pm \omega^2\cr\end{pmatrix}, ~~ \rho(b) = \begin{pmatrix} \pm \omega & 0\cr z& \pm \omega^2\cr\end{pmatrix}, \]
for some $z\in \C$. Consider $\tr(\rho(ab)) = \omega^2 +\omega + z= z-1$. Since $\rho(ab)$ has order $6$, $z-1=\pm\sqrt{3}$ and this gives the two possibilities. One of these gives the faithful discrete representation ($z=1+\sqrt{3}$) and the other gives its Galois conjugate,  a representation into $\mathrm{PSU}(2)$.
\end{proof}

\begin{remark}\label{r:type1} 
In addition to Galois rigidity,  the invariant quaternion algebras of the groups listed in Theorem \ref{arithtriangle} all satisfy Theorem \ref{galois_rigid_real_quad}(ii). To see this, first note that they have type number $1$, by \cite[Prop 3]{Tak}. Since the triangle groups are arithmetic, all the defining quaternion algebras are ramified at one real place, and the finite places where they ramify are the unique places of that characteristic. Briefly, $2$ is the unique ramified place in the case of $d=2,3$ and similarly for $5$ when $d=5$. Also, $2$ and $3$ are inert in the case of $d=5$.
\end{remark}

\begin{remark}\label{rem:notgr} Note that the group $\Delta(4,4,4)$ does admit an infinite representation into $\PSL(2,\C)$ that does not arise as a Galois conjugate representation. In this case the image is the Euclidean triangle group $\Delta(2,4,4)$, which of course is not Zariski dense, and so does not violate Galois rigidity. On the other hand, $\Delta(6,6,6)$ admits an epimorphism to $\Delta(2,6,6)$, and this does violate Galois rigidity. Hence Proposition \ref{trianglereps} does not hold for $\Delta(6,6,6)$.\end{remark}

\subsection{Profinite epimorphisms among Fuchsian groups}\label{s:euler}

By definition, the {\em Euler characteristic} of a Fuchsian group $\G$ is the orbifold Euler characteristic of $\mathbb{H}^2/\G$ (which by Gauss-Bonnet is the area of a fundamental domain for $\G$ divided by $-2\pi$). It behaves multiplicatively on subgroups in the sense that $[\G : H]=d$ implies $\chi(H) = d\chi(\G)$.

We shall need the following lemma, which could be rephrased as saying that if the area of $\mathbb{H}^2/\G_1$ is less than the area of $\mathbb{H}^2/\G_2$, then $\G_2$ has a finite quotient that $\G_1$ does not have.

This is the first place where it is convenient for us to phrase a result in the language of profinite completions. We remind the reader that the {\em profinite completion of a group $\G$} is the inverse limit of its system of finite quotients, $\wh \G = \ilim \G/N$ where the limit is taken over finite index normal subgroups $N<\G$ ordered by reverse-inclusion. It is endowed with the inverse-limit topology, making it a compact topological group. If $\G$ is residually finite (as all of the groups that we consider are) then the natural map $\G\to\wh{\G}$ is injective. The image of $\G$ is dense, and every epimorphism  from $\G$ to a finite group $Q$ extends to a continuous map $\wh \G\to Q$.
 
\begin{lemma}\label{l:euler} 
Let $\G_1$ and $\G_2$ be Fuchsian groups. If there is a continuous surjection $\wh\G_1\to \wh\G_2$, then $\chi(\G_1) \leq \chi(\G_2)$.
\end{lemma}

\begin{proof} 
To begin, given a Fuchsian group $\G$, let $b_1(\G)$ denote the first betti number of $\G$, i.e.~the rank of $H^1(\G, \mathbb{Z})$. If $\G$ is cocompact then $b_1(\G) \le 2- \chi(\G)$, and otherwise, $b_1(\G) \le 1- \chi(\G)$. If additionally $\G$ is torsion-free, then these are equalities. Finally, for finitely generated groups in general, if there an epimorphism $\wh G_1\to \wh G_2$, then $b_1(G_1)\ge b_1(G_2)$ (see \cite[Lemma 2.10]{BCR}).

If there is a continuous epimorphism $\eta\colon \wh\G_1\to \wh\G_2$, then we can pass to a subgroup of finite index, say $H_2<\G_2$ of index $d$, so that $H_2$ is torsion-free. Then $H_1:=\eta^{-1}(\overline{H}_2)\cap\G_1$ is a subgroup of index $d$ in $\G_1$ for which the restriction of $\eta$ to $\overline H_1 \cong \wh H_1$ gives an epimorphism $\wh H_1\to \wh H_2$. (Here, $\overline{H}$ denotes the closure of $H<\G$ in $\wh\G$.) We do not assume that $H_1$ is torsion-free.

Then, 
\[ 2 - d\,\chi(\G_1) \ge b_1(H_1)  \ge b_1(H_2) = \epsilon - d\,\chi(\G_2) \]
where $\epsilon = 1$ or $2$ according to whether $\G_2$ is cocompact or not. As $d$ can be taken to be arbitrarily large, this implies $\chi(\G_1) \le \chi(\G_2)$. 
\end{proof}

\begin{remark}
A  similar proof, invoking L\"uck approximation instead of Euler characteristic, shows that if $\G_1$ and $\G_2$ are finitely presented residually finite groups for which there is a continuous epimorphism $\wh \G_1\to \wh \G_2$, then the first $\ell_2$--betti numbers satisfy $b_1^{(2)}(\G_1)\ge b_1^{(2)}(\G_2)$.
\end{remark}

A stronger form of the following result will be established by less elementary means in Section \ref{s:character}.

\begin{corollary}\label{c:whole}
Let $\D=\Delta(p,q,r)$ with $1>\frac 1 p + \frac 1 q + \frac 1 r \ge \frac 1 2$.
\begin{enumerate}
\item
If $S$ is a non-elementary proper subgroup of $\D$, then there does not exist a continuous epimorphism  $\wh\D\to \wh S$. 
\item 
If $H$ is a non-elementary Fuchsian group that is not cocompact, then does there does not exist a continuous epimorphism  $\wh\D\to \wh H$. 
\end{enumerate}
\end{corollary}

\begin{proof} 
If $S<\D$ is finite index, then $\chi(S) < \chi(\D)$ and the lemma applies.  If not, then we are in Case (2) of the corollary, and we complete the proof using the argument below.

Thus assume that $H$ is a non-elementary Fuchsian group that is not cocompact. Hence it is a free product of a free group and some finite cyclic groups. If there is a surjection $\wh\D\to \wh H$ then $H$ must be finitely generated, and since $\D$ has finite abelianization, $H$ must as well. Thus we are reduced to the possibility that $H$ is a free product of finite cyclic groups $C_1\ast\dots\ast C_m$, where $C_i$ has order $n_i$, say.

Now, 
\[ \chi(\D) = \frac 1 p + \frac 1 q + \frac 1 r -1\ge -\frac 1 2, \]
whereas 
\[ \chi(H) = \frac 1 n_1+  \dots + \frac 1 n_m + 1 - m. \]
So from Lemma \ref{l:euler} we have $\chi(H)\ge-1/2$, which forces $m=2$ or else $m=3$ and $n_1=n_2=n_3=2$. This last possibility has abelianization $(\Z/2\Z)^3$, which cannot be a finite image of $\D$ (and hence $\wh\D$) because $\D$ is generated by two elements. Thus $H = C_1\ast C_2$.

Let $a,b, c=(ab)^{-1}$ be the generators in the standard presentation of $\D$, and note that any two of them suffice to generate. We know from  \cite[Thm 5.1]{BCR} that any finite subgroup of $\wh H$ is contained in a conjugate of $C_1$ or $C_2$, so for any map $\wh\D\to \wh H$, the image of at least two of $a,b,c$  must lie in a conjugate of the same $C_i$. But this means that the image of $\D$  in the abelianization of $H$ would be a proper subgroup. Since $\D$ is dense in $\-\D$, this implies that no continuous map $\wh\D\to \wh H$ can be surjective.
\end{proof}
 
\section{Profinite rigidity of triangle groups}\label{s:triangle_prof_rigid}

In this section we prove Theorem \ref{main}, whose statement we recall for the reader's convenience.  

\main*
 
\begin{proof}  
Let $\Gamma$ be one of the first five groups listed in the theorem and let $\Lambda$ be a finitely generated, residually finite group with $\wh\Lambda \cong \wh\G$. We must prove that $\Lambda\cong\G$.

We saw in Theorem \ref{arithtriangle} that each $\Gamma$ arises as a group $\Gamma_{\mathcal{O}}^1$ which is the image in $\PSL(2,\C)$ of the elements of norm $1$ in a maximal order in a quaternion algebra over a real quadratic field, ramified at one infinite place and one finite place. Proposition \ref{trianglereps} assures us that $\Gamma$ is Galois rigid, so by Theorem \ref{galois_rigid_real_quad} and Remark \ref{r:type1}, there exists an epimorphism $\rho\colon\Lambda \onto L$ onto a finitely generated, Zariski dense (hence non-elementary) subgroup $L<\Gamma$. This induces a continuous epimorphism $\wh\G \cong \wh\Lambda \to\wh{L}$, and Corollary \ref{c:whole} says that this is impossible unless $L=\Gamma$. Thus $\rho$ induces an epimorphism $\hat{\rho}\colon \wh\Lambda \to \wh\Gamma\cong \wh\Lambda$, and since finitely generated, profinite groups are Hopfian (see \cite[Prop 2.5.2]{RZ}), we conclude that $\wh{\rho}$ is injective, and hence $\rho$ is an isomorphism. 

We now deal with $\Gamma=\Delta(4,4,4)$, which arises as a normal subgroup of index $3$ in $\Delta(3,3,4)$. Let $\Lambda$ be a finitely generated, residually finite group with $\wh\Lambda\cong\wh\Gamma$. As above, using Galois rigidity we obtain an epimorphism $\rho\colon \Lambda \onto L$ onto a non-elementary subgroup of $\Delta(3,3,4)=\Gamma_{\mathcal{O}}^1$. Indeed, $L$ must be a subgroup of $\Gamma$, since if not, then $L\cap\Gamma$ is a normal subgroup of index $3$ in $L$, which is impossible as the abelianization of $\Gamma$ and hence $\Lambda$ is $\Z/4\Z$. From Corollary \ref{c:whole} we can now deduce that $L=\Gamma$, and as in the previous case we conclude that $\rho\colon \Lambda\to\Gamma$ is an isomorphism. 

The case of $\Delta(5,5,5)<\Delta(3,3,5)$ is entirely similar to $\Delta(4,4,4)< \Delta(3,3,4)$.

It remains to deal with the seven groups in the bottom row of Theorem \ref{main}. These cases are covered by the discussion in the next section, explicitly Corollary \ref{extending_profinite}, because each contains as a subgroup of index $2$  one of the groups that we have already dealt with. In each case the inclusion $\D(p,p,q)\hookrightarrow \D(2,p,2q)$ is obtained by noting that the hyperbolic isosceles triangle with vertex angles $(\pi/p, \pi/p, \pi/q)$ is divided into two copies of $(\pi/2, \pi/p, \pi/2q)$ by dropping a perpendicular from the vertex with angle $\pi/q$. 
\end{proof}

\begin{remark} In the proof of Theorem \ref{main} we emphasized techniques that rely on basic properties of Fuchsian groups.  Some of this can be bypassed by invoking deeper parts of \cite{BMRS} and Theorem \ref{galois_rigid_real_quad}(i) which implies that $\Lambda$ is Galois rigid, and from this it can deduced that $L$ is necessarily a triangle group (cf. the proof of Corollary \ref{triangle_extra}).\end{remark}

\begin{remark}\label{r:237} 
In \cite[\S 1]{BMRS} we discussed the fact that \cite[Thm 4.8]{BMRS} does not apply to $\D(2,3,7)$, and so we cannot conclude profinite rigidity in this case.  The argument given in the proof of Theorem \ref{main} would apply if we knew that the group $L$ constructed in the proof is actually a subgroup of $\D(2,3,7)$. However, as discussed in \cite[\S 1]{BMRS}, we do not know this: the group $L$ might  be a subgroup of $\PSL(2,R_k)$, where $R_k$ is the ring of integers of the field $\Q(\cos \pi/7)$.  At present we do not know how to exclude this possibility.
\end{remark}

\subsection{Profinite rigidity for index $2$ extensions}

\begin{lemma}\label{nonisoprofinite}
The four groups listed in Proposition \ref{prop_index2}(2) have non-isomorphic profinite completions.
\end{lemma}

\begin{proof} 
We need to prove that the profinite completions of the following groups are distinct
\[ \D(p,q,q)\times \Z/2\Z,\ \D^-(p,q,q),\ \D(2p,q,2),\ \Lambda_\rho(p,q). \] 
We calculate their abelianisations from the explicit presentations given in \S \ref{s:arithtriangle}:
\[ \Z/q\Z\times \Z/h\Z\times \Z/2\Z,\ (\Z/2\Z)^i,\ \Z/h'\Z\times\Z/2\Z,\ \Z/q\Z\times\Z/2\Z, \] 
where $h={\rm{gcd}}(p,q),\, h'={\rm{gcd}}(2p,q)$ and $i\in\{1,2,3\}$ depends on the parity of $p$ and $q$. 

The existence of orientation reversing isometries in $\D^-(p,q,q)$ and $\Lambda$ means that the results of \cite{BCR} do not apply to these groups. However, for $\D$ and $\D(2p,q,2)$ we can appeal to \cite[Thm 5.1]{BCR}, which states that for any Fuchsian group $\G$ the inclusion $\G\to\wh{\G}$ induces a bijection between conjugacy classes of finite subgroups. Thus the maximal finite subgroups of $\wh{\D(2p,q,2)}$ up to conjugacy are  $\Z/{2p}\Z, \Z/q\Z, \Z/2\Z$, while, writing $\D=\Delta(p,q,q)$, those of $\wh{\D\times \Z/2\Z} = \wh\D\times\Z/2\Z$ are $\Z/p\Z  \times \Z/2\Z$ and two copies of  $\Z/q\Z \times \Z/2\Z$. In particular (noting that $q\ge 3$) these two profinite completions are not isomorphic. Moreover, neither is isomorphic to the profinite completion of $\D^-(p,q,q)$ or  $\Lambda_\rho(p,q)$, since these last two groups contain finite dihedral groups. Finally, the profinite completions of $\D^-(p,q,q)$ and  $\Lambda_\rho(p,q)$ are different because the abelianisation of the former is an elementary 2--group while that of the latter is not. 
\end{proof}
  
\begin{corollary}\label{extending_profinite}
If $\D=\D(p,q,r)$ is profinitely rigid in the absolute sense, then so too is any group that contains $\D$ as a subgroup of index $2$.
\end{corollary}
 
We highlight a special case of Corollary \ref{extending_profinite}.

\begin{corollary}\label{c:non-orient}
For each of the groups $\D(p,q,r)$ listed in Theorem \ref{main}, the corresponding Coxeter group $\D^{-}(p,q,r)$ is profinitely rigid in the absolute sense.
\end{corollary}

\section{Additional finite quotients from character varieties}\label{s:character}

In this section we prove Theorem \ref{t:character-var}. The main idea in the proof is to gain control of $\PSL(2,\mathbb F)$ quotients in certain situations (here $\mathbb F$ is a finite field).  This in turn depends on the $(\P)\SL(2,\C)$--character variety. Recall that for any finitely generated group $H$, one has the $\PSL(2,\C)$--representation variety $\Hom(H,\PSL(2,\C))$, and the $\PSL(2,\C)$--character variety $Y(H)$ is the algebro-geometric quotient of $\Hom(H,\PSL(2,\C))$ by the conjugation action. We refer the reader to \cite{BMP} and \cite{BZ} for definitions and further details about the $\PSL(2,\C)$--character variety. 

\subsection{Additional finite quotients}\label{sub_getting}

In this section, we prove Theorem \ref{t:character-var}, that we now restate for the reader's convenience.

\charvar*

Before proving this in \S \ref{proof_extra}, we deduce a corollary of particular interest to us.  

\begin{corollary}\label{triangle_extra}
Let $\Delta(p,q,r)$ be a Fuchsian triangle group, and $H$ a non-elementary Fuchsian group that is not a triangle group.  Then there exists a finite field $\F$ and a representation $H\to\PSL(2,\F)$ whose image is not a quotient group of $\Delta(p,q,r)$.
\end{corollary}

\begin{proof} 
It is well-known that $\Delta(p,q,r)$ has Property FA (see \cite[Example 6.3.5]{Serre}). It is also well-known that if $H$ is any finitely generated Fuchsian group that is not itself a triangle group then, $Y(H)$ contains a positive dimensional  component (the Teichm\"uller component) containing the characters of faithful discrete representations of $H$ in $\PSL(2,\C)$ (see \cite{Ke}) and this proves the corollary.
\end{proof}

\subsection{Proof of Theorem \ref{t:character-var}}\label{proof_extra}

The proof of Theorem \ref{t:character-var} will be completed upon proving the following three lemmas. 

\begin{lemma}\label{finite_reps_incharp}
Let $\G$ be as in Theorem \ref{t:character-var}, $p$ a prime, and ${\mathbb F}$ an algebraic closure of ${\mathbb F}_p$. Then $\G$ has only finitely many  irreducible representations into $\PSL(2,{\mathbb F})$, up to conjugacy.
\end{lemma}

\begin{proof} 
This result is contained in Bass's work on finite $n$--representation type \cite{Bass}, as we now explain. \cite[pg. 59, Prop 22]{Serre} proves that if a finitely generated group $\G$ has Property {\rm{FA}}, then for any representation $\rho\colon \G\to {\rm{GL}}(2,k)$ over any field $k$, the eigenvalues of $\rho(\gamma)$ are integral over $\Z$ for all $\gamma\in \G$: in the terminology of \cite{Bass}, $\G$ has {\em integral $2$--representation type}.  \cite[Prop 5.3]{Bass} and the second remark following it in \cite{Bass} prove the lemma.
\end{proof}

\begin{lemma}\label{algebraicpointoflargedegree}
Let $Y$ be the component defined in Theorem \ref{t:character-var}. Then $Y$ contains a point corresponding to the character of an irreducible representation.
\end{lemma}

\begin{proof}  
Let $\chi_\rho\in Y$ be the character of a faithful, discrete representation of $H$.  Since $Y$ has positive dimension,  $Q=\mathbb{H}^3/\rho(H)$ is non-compact. If $Q$ has finite volume, it must have at least one cusp cross-section that is a torus or $S^2$ with $4$ marked points of cone angle $\pi$. If $Q$ has infinite volume, some component of the boundary of the convex core $C(Q)$ of $Q$ has positive genus or is a copy of $S^2$ with at least $4$ cone points and with at least one cone angle less than $\pi$ (see \cite[Ch 7]{BMP}). In the former case, we can therefore assume that $Y$ is the Dehn surgery component, and so the generic point corresponds to the character of a Zariski dense (and hence irreducible) representation. In the latter case we can assume that $Y$ is a positive dimensional component that contains characters of faithful, discrete representations of $H$ that have geometrically finite image. In particular such representations have Zariski dense image and hence are irreducible.
\end{proof}

\begin{lemma}\label{pos_dim_gives_infinite_reps}
Let $H$ be a finitely generated group for which there exists an irreducible component $Y\subset Y(H)$ of positive dimension containing the character of an irreducible representation of $H$. Then for infinitely many primes $p$, $H$ has infinitely many  conjugacy classes of irreducible representations into $\PSL(2,{\mathbb F})$ where ${\mathbb F}$ is an algebraic closure of ${\mathbb F}_p$.
\end{lemma}

\noindent Some of the ideas used below were motivated by the proofs in \cite[\textsection 8]{Wei}.

\begin{proof}
First, as $Y$ contains the character of an irreducible representation, any representation corresponding to the character of a generic point must also be irreducible. We now fix a representation $\rho\colon H \to \PSL(2,\C)$ whose character is a generic point in $Y$. Let $H_0 = \rho(H)$ be its image, let $H_1 = \phi^{-1}(H_0)$ be the preimage in $\SL(2,\C)$, let $R$ be the subring of $\C$ generated by the $\Z$--span of the traces of $H_1$, and let $K$ be the fraction field of $R$. As $\rho$ is a generic point of the positive dimensional component, $K/\Q$ has positive transcendence degree.

As described in \S \ref{ss:traces}, since $H_1$ is an irreducible subgroup of $\SL(2,\C)$, the $K$--span of $H_1$ is a $K$--quaternion algebra $B/K$ with $B  < \mathrm{M}(2,\C)$. Fixing a quadratic extension $L/K$ which splits $B$ (if $B$ is already split, take $L = K$), we have that $B \otimes_{K} L \cong \mathrm{M}(2,L)$ and $H_1 < \SL(2,L)$. Setting $S$ to be the subring of $L$ generated over $\Z$ by the matrix entries of $H_1$, we see that $H_1 < \SL(2,S)$, $R$ is a subring of $S$, and the field of fractions of $S$ is $L$. As $H_1$ generates $\mathrm{M}(2,L)$, each of the standard basis elements for $\mathrm{M}(2,L)$ can be written as some finite $L$--linear combination of elements in $H_1$. Let $b \in S$ be the product of all the denominators of the elements of $L$ appearing in this linear combination and redefine $S$ by adding $1/b$ to it, so that now, by construction, we have the $S$--span of $H_1$ is $\mathrm{M}(2,S)$. Since $H_1$ is finitely generated, we see that $S$ is finitely generated as an algebra over $\Z$. Finally, since $L$ contains $K$, we see that $L/\Q$ has positive transcendence degree. 

For a given prime $p$, let $\overline{\mathbb{F}_p}$ be the algebraic closure of $\mathbb{F}_p$. Any (non-trivial) ring morphism $\psi\colon S \to \overline{\mathbb{F}_p}$ induces a morphism $\mathrm{M}(2,S) \to \mathrm{M}(2,\psi(S))$ of $S$--modules and  a group homomorphism $\SL(2,S) \to \SL(2,\overline{\mathbb{F}_p})$. As $H_1$ generates $\mathrm{M}(2,S)$ as an $S$--module, it follows that $\psi(H_1)$ generates $\mathrm{M}(2,\psi(S))$ as a $\psi(S)$--module. As $\mathrm{M}(2,\psi(S))$ generates $\mathrm{M}(2,\overline{\mathbb{F}_p})$ over $\overline{\mathbb{F}_p}$, we see that the $\overline{\mathbb{F}_p}$--span of $\psi(H_1)$ is $\mathrm{M}(2,\overline{\mathbb{F}_p})$. Consequently, the induced representation $H_1 \to \SL(2,\overline{\mathbb{F}_p})$ is irreducible. Finally, as $R$ is generated by the traces of $H_1$, if two homomorphisms of $S$ to $\overline{\mathbb{F}_p}$ do not agree on $R$, their induced representations of $H_1$ in $\SL(2,\overline{\mathbb{F}_p})$ will have distinct characters and cannot be conjugate. 

We now will prove that for all but finitely many $p$, the inclusion of $R$ in $S$ induces a function $\mathrm{Hom}(S,\overline{\mathbb{F}_p}) \to \mathrm{Hom}(R,\overline{\mathbb{F}_p})$ with infinite image, which will imply that $H_1$ has infinitely many conjugacy classes of irreducible representations in $\SL(2,\overline{\mathbb{F}_p})$. We will first prove that for all but finitely many primes $\mathrm{Hom}(S,\overline{\mathbb{F}_p})$ is infinite. We will complete the proof by proving that $\mathrm{Hom}(S,\overline{\mathbb{F}_p}) \to \mathrm{Hom}(R,\overline{\mathbb{F}_p})$ is finite-to-one.

We start by proving that $\mathrm{Hom}(S,\overline{\mathbb{F}_p})$ is infinite for all but finitely many primes. As $S$ is finitely generated over $\Z$ and the fraction field of $S$ has positive transcendence degree $r>0$, there exist $y_1,\dots,y_r \in S$ such that the ring morphism $\Z[x_1,\dots,x_r] \to S$ induced by sending $x_i$ to $y_i$ is injective and there exists $n \in \Z$ such that $S[1/n]$ is a finite ring extension of $\Z[y_1,\dots,y_r][1/n]$ \cite[\href{https://stacks.math.columbia.edu/tag/07NA}{Tag 07NA}]{stacks-project}. In particular, $S[1/n]$ is finitely generated as a module over $\Z[y_1,\dots,y_r][1/n]$. 

If $p$ does not divide $n$, then the morphism $\psi \colon \Z[y_1,\dots,y_r] \to \overline{\mathbb{F}_p}$ extends uniquely to a morphism of $\Z[y_1,\dots,y_r][1/n] \to \overline{\mathbb{F}_p}$. We can tensor the ring homomorphism $\Z[y_1,\dots,y_r][1/n] \to S[1/n]$ with $\overline{\mathbb{F}_p}$ along $\psi$ to get an inclusion $\overline{\mathbb{F}_p} \to S[1/n] \otimes \overline{\mathbb{F}_p}$ which gives $S[1/n] \otimes \overline{\mathbb{F}_p}$ as a finite dimensional algebra over $\overline{\mathbb{F}_p}$. As $\overline{\mathbb{F}_p}$ is algebraically closed, $S[1/n] \otimes \overline{\mathbb{F}_p} \cong \prod \overline{\mathbb{F}_p}$ is the product of finitely many copies of $\overline{\mathbb{F}_p}$. Finally, we can project onto one of these factors to get a homomorphism $S \to S[1/n] \to S[1/n] \otimes \overline{\mathbb{F}_p} \to \overline{\mathbb{F}_p}$ which restricts to $\psi$ on $\Z[y_1,\dots,y_r]$. Thus we have shown that each homomorphism of $\Z[y_1,\dots,y_r]$ to $\overline{\mathbb{F}_p}$ extends to at least one ring homomorphism of $S$ to $\overline{\mathbb{F}_p}$. In particular, $\mathrm{Hom}(S,\overline{\mathbb{F}_p})$ is infinite.

We complete the proof by showing the induced function $\mathrm{Hom}(S,\overline{\mathbb{F}_p}) \to \mathrm{Hom}(R,\overline{\mathbb{F}_p})$ is finite-to-one. As $S$ is finitely generated over $\Z$, the inclusion of integral domains $R \hookrightarrow S$ makes $S$ a finitely generated $R$--algebra. So there exist $s_1,\dots,s_m \in S$ such that $S = R[s_1,\dots,s_m]$, and because the extension of fraction fields $L/K$ is finite, each $s_i$ is algebraic over $R$. Thus, if $F$ is any field and $\phi\colon R \to F$ is a ring homomorphism, there are at most finitely many distinct homomorphisms $\phi'\colon S \to F$ which restrict to $R$ to give $\phi$. Hence, the inclusion of $R$ in $S$ induces a function $\mathrm{Hom}(S,F) \to \mathrm{Hom}(R,F)$ which is finite-to-one for any field $F$.
\end{proof}

Although Lemma \ref{pos_dim_gives_infinite_reps} provides $H$ with infinitely many finite quotients that cannot be quotients of a group with Property FA as in Theorem \ref{t:character-var}, they are not explicit. This can be remedied by a more careful analysis of specializations at algebraic points together with Strong Approximation \cite{Wei}. In particular, the finite quotients in question can be taken to be groups of the form $\PSL(2,F)$ for some finite fields $F$, or finite groups $D$ that contain such a $\PSL(2,F)$ as normal subgroups with quotients an elementary abelian $2$--group.




\begin{thebibliography}{9999}

\bibitem{Bass} 
H.~Bass, {\em Groups of integral representation type},  Pacific J. Math. {\bf 86} (1980), 15--51.

\bibitem{BMP} 
M.~Boileau, S.~Maillot and J.~Porti, {\em Three-dimensional orbifolds and their geometric structures}, Panoramas et Synth\`eses, {\bf 15} Soci\'et\'e Math\'ematique de France, (2003).

\bibitem{BZ} 
S.~Boyer and X.~Zhang, {\em On Culler-Shalen seminorms and Dehn filling}, Ann. of Math. {\bf 148} (1998), 737--801.

\bibitem{BCR} 
M.~R.~Bridson, M.~Conder and A.~W.~Reid, {\em Determining Fuchsian groups by their finite quotients},  Israel J. Math. {\bf 214} (2016), 1--41.

\bibitem{BMRS} 
M.~R.~Bridson, D.~B.~McReynolds, A.~W.~Reid, and R.~Spitler, {\em Absolute profinite rigidity and hyperbolic geometry}, \url{https://arxiv.org/abs/1811.04394}.

\bibitem{Ke} 
L.~Keen, {\em Intrinsic moduli on Riemann surfaces},  Ann. of Math. {\bf 84} (1966), 404--420.

\bibitem{Ker} 
S.~P.~Kerckhoff, {\em The Nielsen realization problem}, Ann. of Math. {\bf 117} (1983), 235--265.

\bibitem{MR} 
C.~Maclachlan and A.~W.~Reid, {\em The Arithmetic of Hyperbolic 3-Manifolds}, Springer-Verlag (2003).

\bibitem{RZ} 
L.~Ribes and P.A.~Zalesskii, {\em Profinite Groups}, Springer-Verlag (2000).

\bibitem{Serre} 
J-P.~Serre, {\em Trees}, Springer-Verlag (1980).

\bibitem{stacks-project} 
{\em The Stacks Project}, \url{https://stacks.math.columbia.edu}, (2020).

\bibitem{Tak} 
K.~Takeuchi, {\em Commensurability classes of arithmetic triangle groups}, J. Fac. Sci. Univ. Tokyo {\bf 24} (1977), 201--212.

\bibitem{Wei} 
B.~Weisfeiler, {\em Strong approximation for Zariski dense subgroups of semi-simple algebraic groups}, Ann. of Math. {\bf 120} (1984), 271--315.

\end{thebibliography}
\end{document}